\documentclass[letterpaper]{article}

\usepackage{amsmath,amsthm,bbm,amssymb,latexsym,mathrsfs,amsfonts,titlesec,graphicx,tikz-cd}

\RequirePackage[left=1in,right=1in,top=1in,bottom=1in]{geometry}

\titleformat{\subsection}[runin]
  {\normalfont\bfseries}{\thesubsection}{1em}{}

\newcommand{\Z}{\mathbb{Z}}

\newcommand{\A}{\mathbb{A}}

\newtheorem{theorem}{Theorem}[subsection]
\newtheorem{lemma}[theorem]{Lemma}

\newtheorem{corollary}{Corollary}[theorem]
\newtheorem{proposition}[theorem]{Proposition}
\theoremstyle{definition}
\newtheorem{definition}[theorem]{Definition}

\numberwithin{equation}{subsection}

\DeclareMathOperator{\GL}{GL}

\DeclareMathOperator{\IC}{IC}

\DeclareMathOperator{\Spec}{Spec}

\DeclareMathOperator{\Map}{Map}
\DeclareMathOperator{\Bun}{Bun}

\DeclareMathOperator{\pos}{pos}
\DeclareMathOperator{\disj}{disj}

\DeclareMathOperator{\pt}{pt \!}

\begin{document}

\title{A resolution of singularities for Drinfeld's compactification by stable maps}
\author{Justin Campbell}
\maketitle

\begin{abstract}

Drinfeld's relative compactification plays a basic role in the theory of automorphic sheaves, and its singularities encode representation-theoretic information in the form of intersection cohomology. We introduce a resolution of singularities consisting of stable maps from nodal deformations of the curve into twisted flag varieties. As an application, we prove that the twisted intersection cohomology sheaf on Drinfeld's compactification is universally locally acyclic over the moduli stack of $G$-bundles at points sufficiently antidominant relative to their defect.

\end{abstract}

\section{Introduction}

\subsection{} Fix an algebraically closed field $k$. Let $G$ be a connected reductive group over $k$, and $X$ a connected smooth projective curve over $k$. We assume for simplicity that the derived group of $G$ is simply connected. Fix a Borel subgroup $B \subset G$. For any algebraic group $H$ we write $\Bun_H$ for the moduli stack of $H$-bundles on $X$. Write $\Lambda$ for the lattice of coweights, which parameterize the connected components of $\Bun_B$.

\subsection{} The purpose of this paper is to introduce a resolution of singularities for Drinfeld's relative compactification $\overline{\Bun}_B \to \Bun_G$. In the appendix to \cite{FFKM} such a resolution is constructed, when $X$ has genus zero, for the space of quasimaps, which is the fiber of Drinfeld's compactification over the trivial $G$-bundle. Their resolution parameterizes stable maps (in the sense of Kontsevich, see \cite{Kon}) into the flag variety. When the $G$-bundle is allowed to vary, we obtain a relative compactification $\overline{\Bun}\!\,^K_B \to \Bun_G$ mapping to $\overline{\Bun}_B$ over $\Bun_G$, which we call Kontsevich's compactification. The stack $\overline{\Bun}\!\,^K_B$ is remarkably well-behaved: it is smooth for $X$ of arbitrary genus, and the complement of $\Bun_B$ is a normal crossings divisor (see Propositions 2.4.1 and 4.4.1 below).

In the case $G = \GL_n$ Laumon's compactification resolves the singularities of $\overline{\Bun}_B$ (see e.g. \cite{Kuz}), but it is unclear how to generalize this construction to arbitrary reductive $G$. The approach via stable maps looks completely different: whereas Laumon's compactification consists of quasimaps into a smooth stack, Kontsevich's compactification allows the curve to vary and to acquire nodes.

\subsection{} We use Kontsevich's compactification to prove that the twisted intersection cohomology of $\overline{\Bun}_B$ is universally locally acyclic over $\Bun_G$ at points sufficiently antidominant relative to their defect: this is Theorem \ref{ulatwist} below. This is the assertion of Proposition 7.9 in Gaitsgory's work \cite{G} on the quantum geometric Langlands correspondence. A different argument was suggested in \emph{loc. cit.}, but the author was unable to carry it out. Instead, we will deduce the local acyclicity from the smoothness of the projection $\overline{\Bun}\!\,^K_B \to \Bun_G$ at points sufficiently antidominant relative to their defect. Some additional care is required in the twisted case, with the normal crossings property of the boundary playing a crucial role.

In \cite{C} the author will apply the local acyclity theorem in the untwisted case to describe the nearby cycles sheaf of a degenerating one-parameter family of Whittaker sheaves. The twisted version will be applied in S. Lysenko's forthcoming work extending \cite{G}.

\subsection*{} \emph{Acknowledgements.} I thank Dennis Gaitsgory for suggesting that I use stable maps to prove Theorem \ref{ulatwist}, and for patiently answering many questions. I also thank Sergey Lysenko for careful reading and for comments which substantially improved the paper.

\section{The Kontsevich compactification}

\subsection{} Recall that a map $f : C \to Y$ from a nodal curve to a variety $Y$ is called \emph{stable} if $C$ admits no first-order infinitesimal automorphisms along which $f$ is constant. Said differently, the automorphism group of $f$ in the moduli of curves mapping to $Y$ is finite. Concretely, this means that if $C_0 \subset C$ is a rational component on which $f$ is constant, then $C_0$ contains at least three nodes of $C$, with a similar condition on genus one components which will not be relevant.

\subsection{} In what follows, the letter $S$ will always denote an affine scheme.

\begin{definition}

The \emph{Kontsevich compactification} $\overline{\Bun}\!\,^K_B$ is defined as follows. A point $S \to \overline{\Bun}\!\,^K_B$ is a commutative square \[
\begin{tikzcd}
C \ar{d}{p} \ar{r} & \pt/B \ar{d} \\
X \times S \ar{r}{\mathscr{P}_G} & \pt/G,
\end{tikzcd} \]
where $C \to S$ is a flat family of connected nodal projective curves of arithmetic genus $g$, the map $p$ has degree $1$, and the induced map $C \to (G/B)_{\mathscr{P}_G}$ is stable over every geometric point of $S$.

\end{definition}

For a geometric point $\Spec k \to \overline{\Bun}\!\,^K_B$, the projection $p : C \to X$ restricts to an isomorphism on some component, i.e. it admits a canonical section $X \to C$. Such a section does not exist for a general family. Let $\mathring{X}$ be the intersection of the smooth locus of $C$ with $X$. Observe that the connected components of $C \setminus \mathring{X}$ have arithmetic genus zero and hence are simply connected. In particular, the dual graph of $C$ is a tree, and all irreducible components not equal to $X$ are isomorphic to $\mathbb{P}^1$.

\subsection{} The open locus in $\overline{\Bun}\!\,^K_B$ where $p$ is an isomorphism (equivalently, where $C$ is smooth over $S$) identifies with $\Bun_B$, and the projection $\overline{\mathfrak{p}}\!\,^K : \overline{\Bun}\!\,^K_B \to \Bun_G$ extends $\Bun_B \to \Bun_G$. The degree of a point in $\overline{\Bun}\!\,^K_B$ is defined to be the degree of the $B$-bundle, and we denote the component of degree $\lambda \in \Lambda$ by $\overline{\Bun}\!\,^{K,\lambda}_B$.

\begin{proposition}

The stack $\overline{\Bun}\!\,^K_B$ is Artin and locally of finite type, and for any $\lambda \in \Lambda$ the projection \[ \overline{\mathfrak{p}}\!\,^{K,\lambda} : \overline{\Bun}\!\,^{K,\lambda}_B \to \Bun_G \] is proper (but not schematic in general).

\end{proposition}

\begin{proof}

For any point $S \to \Bun_G$ the fiber product $\overline{\Bun}\!\,^K_B \times_{\Bun_G} S$ is the open (and closed) locus in the stack $\overline{\mathscr{M}}_g((G/B)_{\mathscr{P}_G})$ of stable maps defined on arithmetic genus $g$ curves where the composition with $(G/B)_{\mathscr{P}_G} \to X \times S$ has degree one. It is well-known that $\overline{\mathscr{M}}_g((G/B)_{\mathscr{P}_G})$ is representable by a Deligne-Mumford stack locally of finite type (see \cite{AO}), from which it follows that $\overline{\Bun}\!\,^K_B$ is an Artin stack locally of finite type. Moreover, in \emph{loc. cit.} it is shown that $\overline{\mathscr{M}}_g((G/B)_{\mathscr{P}_G}) \to S$ is proper when the degree is fixed. 

\end{proof}

\subsection{} Denote by $\mathscr{M}_X$ the moduli stack of connected nodal curves $C$ equipped with a proper degree one map $p : C \to X$ satisfying $R^1p_*\mathscr{O}_C = 0$. Note that in this definition $X$ need not be proper. If $X$ is proper, then $\mathscr{M}_X$ can also be described as the moduli stack of proper connected nodal curves of arithmetic genus $g$ equipped with a degree one map to $X$.

\begin{proposition}

Both $\mathscr{M}_X$ and the projection $\pi : \overline{\Bun}\!\,^K_B \to \mathscr{M}_X$ are smooth, so $\overline{\Bun}\!\,^K_B$ is smooth.

\label{kontsm}

\end{proposition}

\begin{proof}

First we prove that $\mathscr{M}_X$ is smooth. Observe that \[ T_p(\mathscr{M}_X) = R\Gamma(C,T(C/X))[1], \] so we must prove that $R^i\Gamma(C,T(C/X)) = 0$ for $i > 1$. A simple local calculation shows that $T(C)$ is concentrated in degrees $0$ and $1$, and that $H^1(T(C))$ is the sum of the delta sheaves at the nodes. It follows that $R^i\Gamma(C,T(C/X)) = 0$ for $i > 2$, so it remains to check that \[ R^1\Gamma(C,T(C)) \longrightarrow R^1\Gamma(C,p^*T(X)) \] is surjective.

Applying Serre duality, we must show that \[ R^0\Gamma(C,\omega_C \otimes p^*T^*(X)) \longrightarrow R^0\Gamma(C,\omega_C \otimes T^*(C)) \] is injective, where here $\omega_C = \det T^*(C)$ is the dualizing line bundle. We claim that the composition \[ R^0\Gamma(C,\omega_C \otimes p^*T^*(X)) \longrightarrow R^0\Gamma(C,\omega_C \otimes T^*(C)) \longrightarrow R^0(\mathring{X},T^*(\mathring{X})^{\otimes 2}) \] is injective, where the second map is restriction to $\mathring{X}$, the intersection of $X$ with the smooth locus of $C$. Observe that for any irreducible component $C_0$ other than $X$ we have \[ (\omega_C \otimes p^*T^*(X))|_{C_0} \cong \omega_{C_0}. \] Since $\omega_{C_0}$ has no global sections, any global section of the line bundle $\omega_C \otimes p^*T^*(X)$ is supported on $X$, which implies the desired injectivity.

To prove smoothness of $\pi$, fix a geometric point $\xi : \Spec k \to \overline{\Bun}\!\,^K_B$ over $p : C \to X$. The fiber $\pi^{-1}(p)$ maps to $\Map(C,\pt/B)$, which is smooth because $C$ is a curve. The square \[
\begin{tikzcd}
\pi^{-1}(p) \ar{d} \ar{r} & \Bun_G \ar{d} \\
\Map(C,\pt/B) \ar{r} & \Map(C,\pt/G),
\end{tikzcd} \]
is not Cartesian, but $\pi^{-1}(p)$ is the open locus in the fiber product defined by the stability condition. It therefore suffices to prove that the map $p^* : \Bun_G \to \Map(C,\pt/G)$ is smooth at $\mathscr{P}_G$. Since both domain and target are smooth stacks, we need only show that \[ R^1\Gamma(X,\mathfrak{g}_{\mathscr{P}_G}) \longrightarrow R^1\Gamma(C,p^*(\mathfrak{g}_{\mathscr{P}_G}))
\] is surjective.

This is equivalent to the claim that
\begin{equation}
R^0\Gamma(C,\omega_C \otimes p^*(\mathfrak{g}^*_{\mathscr{P}_G})) \longrightarrow R^0\Gamma(X,\omega_X \otimes \mathfrak{g}^*_{\mathscr{P}_G})
\label{injclaim}
\end{equation}
is injective. Recall that $C$ is Gorenstein, i.e. $\omega_C$ is a line bundle. A global section of $\omega_C \otimes p^*(\mathfrak{g}^*_{\mathscr{P}_G})$ vanishes on any irreducible component $C_0 \neq X$, since the restriction of the vector bundle there is isomorphic to $\mathfrak{g}^* \otimes \omega_{C_0}$ and $C_0 \cong \mathbb{P}^1$. Thus both sides of (\ref{injclaim}) inject into $R^0\Gamma(\mathring{X},\omega_X \otimes \mathfrak{g}^*_{\mathscr{P}_G})$, which finishes the proof.

\end{proof}

\section{The morphism to Drinfeld's compactification}

\subsection{} We recall the definition of Drinfeld's compactification $\overline{\Bun}_B$. Let $\check{\Lambda}$ be the weights of $G$ and denote the dominant weights by $\check{\Lambda}^+$. A map $S \to \overline{\Bun}_B$ is the following data: a $G$-bundle $\mathscr{P}_G$ and a $T$-bundle $\mathscr{P}_T$ on $X \times S$, and for every $\check{\lambda} \in \check{\Lambda}^+$ a monomorphism of coherent sheaves \[ \kappa^{\check{\lambda}} : \check{\lambda}(\mathscr{P}_T) \longrightarrow V^{\check{\lambda}}_{\mathscr{P}_G}, \] satisfying Pl\"{u}cker-type relations (see \cite{BG} for details).

Constructing the morphism $\overline{\Bun}\!\,^K \to \overline{\Bun}_B$ requires some preliminary observations.

\begin{lemma}

The stack $\Bun_B$ is dense in $\overline{\Bun}\!\,^K_B$. In particular, if $S \to \overline{\Bun}\!\,^K_B$ is smooth then the resulting map $p : C \to X \times S$ is an isomorphism over a locus of codimension at least $2$ in $X \times S$.

\label{dense}

\end{lemma}

\begin{proof}

Density follows from the fact that the complement of $\Bun_B$ in $\overline{\Bun}\!\,^K_B$ is a normal crossings divisor: combine Proposition \ref{kontsm} and Proposition \ref{ncd} below.

Notice that the locus $D \subset X \times S$ over which $p$ is not an isomorphism is finite over $S$. If $S \to \overline{\Bun}\!\,^K_B$ is smooth then the image of $D$ in $S$ in a proper subset, whence the claim.

\end{proof}

\subsection{} Fix a morphism $\xi : S \to \overline{\Bun}\!\,^K_B$ with $S$ an affine scheme. To define the desired map $S \to \overline{\Bun}_B$ we use the $G$-bundle $\mathscr{P}_G$ on $X \times S$ obtained from $\xi$, so we still need to specify a $T$-bundle $\mathscr{P}_T$ and, for every $\check{\lambda} \in \check{\Lambda}^+$, a monomorphism of coherent sheaves \[ \kappa^{\check{\lambda}} : \check{\lambda}(\mathscr{P}_T) \longrightarrow V^{\check{\lambda}}_{\mathscr{P}_G} \] satisfying the Pl\"{u}cker relations.

The point $\xi$ also includes a family of curves $C \to S$ with a degree one map $p : C \to X \times S$ and a $B$-reduction \[ G \overset{B}{\times} \mathscr{P}_B \tilde{\longrightarrow} p^*\mathscr{P}_G, \] which can be viewed as a collection of vector bundle embeddings \[ \iota^{\check{\lambda}} : \check{\lambda}(\mathscr{P}_B) \longrightarrow p^*V^{\check{\lambda}}_{\mathscr{P}_G} \] for $\check{\lambda} \in \check{\Lambda}^+$.

\begin{proposition}

For every $\check{\lambda} \in \check{\Lambda}$ the direct image $p_*\check{\lambda}(\mathscr{P}_B)$ is a perfect complex on $X \times S$, so that its determinant line bundle $\det(p_*\check{\lambda}(\mathscr{P}_B))$ is well-defined. There exists a unique $T$-bundle $\mathscr{P}_T$ on $X \times S$ such that $\check{\lambda}(\mathscr{P}_T) = \det(p_*\check{\lambda}(\mathscr{P}_B))$ for all $\check{\lambda} \in \check{\Lambda}$, and for all $\check{\lambda} \in \check{\Lambda}^+$ the embeddings $\iota^{\check{\lambda}}$ canonically induce monomorphisms of coherent sheaves \[ \kappa^{\check{\lambda}} : \check{\lambda}(\mathscr{P}_T) \longrightarrow V^{\check{\lambda}}_{\mathscr{P}_G} \] satisfying the Pl\"{u}cker relations.

\label{resprop}

\end{proposition}

\begin{proof}

By a standard descent argument we reduce to the case that $\xi : S \to \overline{\Bun}\!\,^K_B$ is smooth. Then $S$ is smooth by Proposition \ref{kontsm}, so $p_*\check{\lambda}(\mathscr{P}_B)$ is a perfect complex because $p$ is proper and $X \times S$ is smooth. To show that $\mathscr{P}_T$ exists, we claim that the isomorphisms \[ (\check{\lambda} + \check{\mu})(\mathscr{P}_B) \tilde{\longrightarrow} \check{\lambda}(\mathscr{P}_B) \otimes \check{\mu}(\mathscr{P}_B) \] on $C$ canonically give rise to isomorphisms \[ \det(p_*(\check{\lambda} + \check{\mu})(\mathscr{P}_T)) \tilde{\longrightarrow} \det(p_*\check{\lambda}(\mathscr{P}_T)) \otimes \det(p_*\check{\mu}(\mathscr{P}_T)) \] on $X \times S$. By Lemma \ref{dense}, these isomorphisms are defined away from a locus of codimension at least $2$, whence they extend uniquely to $X \times S$. Similarly, apply Lemma \ref{dense} to obtain the $\kappa^{\check{\lambda}}$ from the $\iota^{\check{\lambda}}$. The Pl\"{u}cker relations on $X \times S$ are inherited from those on $C$ by construction.

\end{proof}

We summarize the construction of the map as follows.

\begin{proposition}

The construction in Proposition \ref{resprop} defines a proper morphism $\overline{\Bun}\!\,^K_B \to \overline{\Bun}_B$ which restricts to the identity on $\Bun_B$ and commutes with the projections to $\Bun_G$.

\end{proposition}

\begin{proof}

Since $\det$ commutes with inverse image, the construction defines a morphism of stacks by base change for quasi-coherent sheaves. It clearly restricts to the identity on $\Bun_B$ and commutes with the projections to $\Bun_G$. Since $\overline{\Bun}\!\,^{K,\lambda}_B \to \Bun_G$ is proper, so is $\overline{\Bun}\!\,^{K,\lambda}_B \to \overline{\Bun}_B$. Proposition \ref{kontdef} below shows that $\overline{\Bun}\!\,^K_B \to \overline{\Bun}_B$ preserves degree and hence is proper.

\end{proof}

\subsection{} The defect of a geometric point $\xi : \Spec k \to \overline{\Bun}\!\,^K_B$ is the $\Lambda^{\pos}$-weighted effective divisor defined as follows. Let $C_1,\cdots,C_n$ be the connected components of $C \setminus \mathring{X}$ and $\{ x_i \} = C_i \cap X$. Then the defect of $\xi$ is \[ -\sum_{i=1}^n \deg(\mathscr{P}_B|_{C_i}) \cdot x_i. \] For any $\mu \in \Lambda^{\pos}$ we denote by $\overline{\Bun}\!\,^K_{B,\leq \mu}$ the open locus where \[ -\sum_{i=1}^n \deg(\mathscr{P}_B|_{C_i}) \leq \mu. \]

\begin{proposition}

The image of $\xi : \Spec k \to \overline{\Bun}\!\,^K_B$ under $\overline{\Bun}\!\,^K_B \to \overline{\Bun}_B$ has the same defect as $\xi$, and its saturation is obtained by restricting the $B$-reduction on $C$ along the section $s : X \to C$. In particular $\overline{\Bun}\!\,^K_B \to \overline{\Bun}_B$ preserves degree.

\label{kontdef}

\end{proposition}

\begin{proof}

Suppose that $n=1$: the general case will follow by a straightforward induction. Fix $\check{\lambda} \in \check{\Lambda}^+$ and put $\mathscr{L} = \check{\lambda}(\mathscr{P}_B)$. Consider the short exact sequence \[ 0 \longrightarrow \mathscr{L} \longrightarrow s_*s^*(\mathscr{L}) \oplus i_*i^*(\mathscr{L}) \longrightarrow \delta_{x_1} \longrightarrow 0 \] on $C$, where $i : C_1 \to C$ is the inclusion. Pushing forward to $X$, we obtain the exact triangle \[ p_*\mathscr{L} \longrightarrow s^*(\mathscr{L}) \oplus (\delta_{x_1} \otimes R\Gamma(C_1,i^*\mathscr{L})) \longrightarrow \delta_{x_1}. \] The claim is that $\det(Rp_*\mathscr{L}) \tilde{\to} s^*(\mathscr{L})(d \cdot x_1)$, where $d := \deg(i^*\mathscr{L}) \leq 0$.

Suppose that $i^*\mathscr{L}$ is nontrivial. Since $s^*\mathscr{L} \to \delta_{x_1}$ is an epimorphism, we get a short exact sequence \[ 0 \longrightarrow R^0p_*\mathscr{L} \longrightarrow s^*\mathscr{L} \longrightarrow \delta_{x_1} \longrightarrow 0 \] and an isomorphism \[ R^1p_*\mathscr{L} \tilde{\longrightarrow} \delta_{x_1} \otimes R^1\Gamma(C_1,i^*\mathscr{L}). \] Thus $R^0p_*\mathscr{L} = s^*(\mathscr{L})(-x_1)$ and $\det(R^1p_*\mathscr{L}) = \mathscr{O}_X((-d-1) \cdot x_1)$. The monomorphism of coherent sheaves associated with $\Spec k \to \overline{\Bun}\!\,^K_B \to \overline{\Bun}_B$ and $\check{\lambda}$ is the composition \[ \det(p_*\mathscr{L}) \tilde{\longrightarrow} s^*(\mathscr{L})(d \cdot x) \longrightarrow s^*\mathscr{L} \longrightarrow s^*p^*V^{\check{\lambda}}_{\mathscr{P}_G} = V^{\check{\lambda}}_{\mathscr{P}_G}, \] and since the last map is a vector bundle embedding the claim follows in this case.

If $i^*\mathscr{L}$ is trivial (this can only occur if the semisimple rank of $G$ is at least $2$), then we have instead a short exact sequence \[ 0 \longrightarrow R^0p_*\mathscr{L} \longrightarrow s^*(\mathscr{L}) \oplus \delta_{x_1} \longrightarrow \delta_{x_1} \longrightarrow 0 \] and $R^1p_*\mathscr{L} = 0$. Therefore $R^0p_*\mathscr{L} \tilde{\to} s^*\mathscr{L}$, so the claim is immediate in this case.

\end{proof}

\section{Relative smoothness and local acyclicity} 

\label{ulasec}

\subsection{} Now we show that the Kontsevich compactification is smooth over $\Bun_G$ at points sufficiently antidominant relative to their defect.

\begin{proposition}

Fix $\mu \in \Lambda^{\pos}$. If $\lambda \in \Lambda$ satisfies $\langle \check{\alpha},\lambda - \mu' \rangle > 2g-2$ for all negative roots $\check{\alpha}$ and all $0 \leq \mu' \leq \mu$, then the morphism \[ \overline{\Bun}\!\,^{K,\lambda}_{B,\leq \mu} \longrightarrow \mathscr{M}_X \times \Bun_G \] is smooth. In particular, the composition with $\mathscr{M}_X \times \Bun_G \to \Bun_G$ is smooth.

\label{relsm}

\end{proposition}

\begin{proof}

The fiber over $p : C \to X$ and $\mathscr{P}_G$ is the open locus in $\Map_X(C,(G/B)_{\mathscr{P}_G})$ consisting of stable maps. Fixing a point $f : C \to (G/B)_{\mathscr{P}_G}$ in the fiber, we have \[ T_f(\Map_X(C,(G/B)_{\mathscr{P}_G})) = R\Gamma(C,f^*T((G/B)_{\mathscr{P}_G}/X)). \] It therefore suffices to prove that \[ R^1\Gamma(C,f^*T((G/B)_{\mathscr{P}_G}/X)) = 0 \] when the degree $\lambda$ of $f$ (i.e. the degree of the resulting $B$-bundle $\mathscr{P}_B$ on $C$) satisfies the given hypothesis.

One can check that \[ f^*T((G/B)_{\mathscr{P}_G}/X) \tilde{\longrightarrow} (\mathfrak{g}/\mathfrak{b})_{\mathscr{P}_B}. \] The latter vector bundle can be filtered by subbundles so that the associated graded is $(\mathfrak{n}^-)_{\mathscr{P}_T}$, where $\mathscr{P}_T$ is the $T$-bundle induced by $\mathscr{P}_B$. Let $\lambda_X = \deg(\mathscr{P}_B|_X)$ and note that $\lambda - \lambda_X \leq \mu$. Thus the hypothesis on $\lambda$ implies that \[ \langle \lambda_X,\check{\alpha} \rangle = \langle \lambda,\check{\alpha} \rangle - \langle \lambda - \lambda_X,\check{\alpha} \rangle > 2g-2 \] for any negative root $\check{\alpha}$, whence $R^1\Gamma(X,\check{\alpha}(\mathscr{P}_T)) = 0$.

Now we induct on the number of components in $C$. Fix a leaf $C_0$ of the tree $C$, i.e. an irreducible component other than $X$ which contains a single node $c$, and put $C' := C \setminus (C_0 \setminus \{ c \})$. Then we have a short exact sequence \[ 0 \longrightarrow (\mathfrak{g}/\mathfrak{b})_{\mathscr{P}_B}|_{C_0}(-c) \longrightarrow (\mathfrak{g}/\mathfrak{b})_{\mathscr{P}_B} \longrightarrow (\mathfrak{g}/\mathfrak{b})_{\mathscr{P}_B}|_{C'} \longrightarrow 0. \] By the inductive hypothesis it suffices to prove that $R^1\Gamma(C_0,(\mathfrak{g}/\mathfrak{b})_{\mathscr{P}_B}|_{C_0}(-c)) = 0$. This follows from the observation that $(\mathfrak{g}/\mathfrak{b})_{\mathscr{P}_B}|_{C_0}$ is generated by global sections.

\end{proof}

We view the intersection cohomology sheaf on $\overline{\Bun}_B$ as a holonomic $\mathscr{D}$-module when $k$ has characteristic zero, or an $\ell$-adic sheaf for a prime $\ell$ not equal to the characteristic of $k$.

\begin{corollary}

For $\lambda,\mu \in \Lambda$ as in Proposition \ref{relsm}, $\IC_{\overline{\Bun}_{B,\leq \mu}^{\lambda}}$ is universally locally acyclic over $\Bun_G$.

\label{ula}

\end{corollary}

\begin{proof}

The pushforward of $\IC_{\overline{\Bun}\!\,^K_B}$ along $\overline{\Bun}\!\,^K_B \to \overline{\Bun}_B$ is semisimple by Theorem 8.3.2.4 of \cite{S}, and since this morphism is birational it follows that the pushforward contains $\IC_{\overline{\Bun}_B}$ as a summand. The proposition implies that $\IC_{\overline{\Bun}\!\,^{K,\lambda}_{B,\leq \mu}}$ is ULA over $\Bun_G$ for $\lambda,\mu$ as above, and now the claim follows because the ULA condition is preserved under proper pushforward and taking summands.

\end{proof}

\subsection{} In the remainder of this section we show that Corollary \ref{ula} still holds after we introduce a twist. A twisting parameter is a multiplicative local system on $\mathbb{G}_m$, so if $k$ has characteristic zero then twistings are in bijection with $k/\Z$.

Let $\mathscr{L}_{\Bun_G}$ be the determinant line bundle on $\Bun_G$, and let $\mathscr{L}_{\Bun_T}$ be the line bundle on $\Bun_T$ whose fiber at $\mathscr{P}_T$ is \[ \det R\Gamma(X,\mathfrak{h} \otimes \mathscr{O}_X) \otimes \bigotimes_{\check{\alpha}} \det R\Gamma(X,\check{\alpha}(\mathscr{P}_T)), \] where the tensor product runs over the roots of $G$. The twisting line bundle on $\overline{\Bun}_B$, which appears naturally in the theory of Eisenstein series, is \[ \mathscr{L}_{\overline{\Bun}_B} := \overline{\mathfrak{p}}^*(\mathscr{L}_{\Bun_G})^{-1} \otimes \overline{\mathfrak{q}}^*(\mathscr{L}_{\Bun_T}), \] where $\overline{\mathfrak{p}} : \overline{\Bun}_B \to \Bun_G$ and $\overline{\mathfrak{q}} : \overline{\Bun}_B \to \Bun_T$ are the natural maps.

Observe that $\mathscr{L}_{\overline{\Bun}_B}$ trivializes canonically over $\Bun_B$. This is because for a fixed $\mathscr{P}_B$ with induced $G$-bundle $\mathscr{P}_G$ and $T$-bundle $\mathscr{P}_T$, the adjoint bundle $\mathfrak{g}_{\mathscr{P}_G}$ admits a filtration whose subquotients are the line bundles $\check{\alpha}(\mathscr{P}_T)$ and the trivial vector bundle $\mathfrak{h} \otimes \mathscr{O}_X$.

For any twisting $c$, the corresponding category of twisted sheaves $D^c(\overline{\Bun}_B)$ is defined as the category of sheaves on the total space of $\mathscr{L}_{\overline{\Bun}_B}$ which are $\mathbb{G}_m$-equivariant against the character $c$ (here ``sheaves" can be interpreted either as $\ell$-adic sheaves or $\mathscr{D}$-modules, depending on the characteristic of $k$). Since $\mathscr{L}_{\overline{\Bun}_B}$ trivializes canonically over $\Bun_B$, we obtain an equivalence \[ D^c(\Bun_B) \tilde{\longrightarrow} D(\Bun_B). \] Under this equivalence we can view $\IC_{\Bun_B}$ as an object of $D^c(\Bun_B)$, and we denote by $\IC^c_{\overline{\Bun}_B}$ its intermediate extension in $D^c(\overline{\Bun}_B)$.

\begin{theorem}

For $\lambda,\mu \in \Lambda$ as in Proposition \ref{relsm}, $\IC^c_{\overline{\Bun}_{B,\leq \mu}^{\lambda}}$ is universally locally acyclic over $\Bun_G$.

\label{ulatwist}

\end{theorem}

The proof will occupy the rest of this section.

\subsection{} We will deduce Theorem \ref{ulatwist} from the following key lemma. Let $f : \mathscr{Y} \to \mathscr{Z}$ be a morphism of smooth Artin stacks, and let $D \subset \mathscr{Y}$ be a normal crossings divisor. Suppose we are given a line bundle of the form
\begin{equation}
\mathscr{O}_{\mathscr{Y}}(\textstyle \sum_i r_iD_i),
\label{twistmult}
\end{equation}
where the $D_i$ are the irreducible components of $D$ and $r_i \in \Z$. There is a corresponding category of twisted sheaves $D^c(\mathscr{Y})$.  Since the line bundle trivializes canonically over $\mathscr{Y} \setminus D$, we can consider the twisted IC sheaf $\IC^c_{\mathscr{Y}}$.

The normal crossings divisor $D$ gives rise to a stratification on $\mathscr{Y}$. Namely, for $n \geq 0$ let $\mathscr{Y}_n \subset \mathscr{Y}$ be the codimension $n$ stratum, which is the locus where $D$ is locally isomorphic to an intersection of $n$ coordinate hyperplanes in affine space.

\begin{lemma}

If $\mathscr{Y}_n \to \mathscr{Z}$ is smooth for all $n \geq 0$, then $\IC^c_{\mathscr{Y}}$ is universally locally acyclic over $\mathscr{Z}$.

\label{keylem}

\end{lemma}

\begin{proof}

Since the ULA property is local in the smooth topology on the source and target, the assertion reduces to the case that $Y = \mathscr{Y}$ and $Z = \mathscr{Z}$ are schemes. Let $C \subset T^*(Y)$ be the singular support of $\IC^c_Y$. By the characterization of singular support introduced in \cite{B}, it suffices to prove that for any point $y \in Y(k)$, we have $(df_y)^{-1}(C_y) \setminus \{ 0 \} = \varnothing$.

We may assume that the multiplicities $r_i$ in (\ref{twistmult}) are all nonzero. If $c$ is trivial, then $C$ is the zero section. Since the hypothesis implies that $f$ is smooth, the claim follows in this case. If $c$ is nontrivial, then for $y \in Y_n$ the singular support $C_y \subset T_y^*(Y)$ is the union of the $n$ lines orthogonal to the $n$ hyperplanes in $T_y(Y)$ determined by the normal crossings property of $D$. The span of $C_y$ is precisely the kernel of $T^*_y(Y) \to T^*_y(Y_n)$, so the injectivity of $T^*_{f(y)}(Z) \to T^*_y(Y_n)$ implies the claim in this case also.

\end{proof}

\subsection{} In order to apply Lemma \ref{keylem}, we must show that the boundary of the Kontsevich compactification has normal crossings. According to Proposition \ref{kontsm}, we need only prove the corresponding claim for $\mathscr{M}_X$.

If $f : X \to Y$ is an \'{e}tale map of smooth curves, there is a corresponding morphism $\mathscr{M}_Y \to \mathscr{M}_X$ which sends $C \to Y \times S$ to $X \times_Y C \to X \times S$. It is \'{e}tale at a $k$-point $C \to Y$ if and only if for all nodes $y_1,\cdots,y_n$ of $C$ contained in $Y$, the fiber $f^{-1}(y_i)$ is a single point.

\begin{proposition}

The complement $\partial \mathscr{M}_X$ of the open point in $\mathscr{M}_X$ given by $X$ is a normal crossings divisor.

\label{ncd}

\end{proposition}

\begin{proof}

Fix a $k$-point $p : C \to X$ of $\partial \mathscr{M}_X$. The claim reduces to the case that $X = \A^1$ as follows. Choose an open $U \subset X$ which contains all the nodes $x_1,\cdots,x_n$ of $C$ contained in $X$, and such that there exists an \'{e}tale function $f : U \to \A^1$ with the property that $f^{-1}(f(x_i)) = \{ x_i \}$ for all $1 \leq i \leq n$. Then we obtain maps $\mathscr{M}_X \to \mathscr{M}_U \leftarrow \mathscr{M}_{\A^1}$, and the image of $p$ in $\mathscr{M}_U$ lifts to $\tilde{p}$ in $\mathscr{M}_{\A^1}$. Moreover, $\mathscr{M}_X \to \mathscr{M}_U$ is \'{e}tale and $\mathscr{M}_{\A^1} \to \mathscr{M}_U$ is \'{e}tale at $\tilde{p}$.

For $X = \A^1$ we will explicitly construct a smooth chart around $p$. There is a canonical stratification \[ X = C_0 \subset C_1 \subset \cdots \subset C_m = C \] defined as follows: the subgraph of the dual graph corresponding to $C_i$ contains the vertices which are connected to the vertex $X$ by a path consisting of at most $i$ edges. We will use the notation $\A^r_{\disj}$ to denote the open locus in $\A^r$ where no two coordinates are equal.

Define $S_1 := \A^n_{\disj} \times \A^n$. For any $1 \leq j \leq n$ let \[ Z_{1j} := \{ (z,(z_i),(t_i)) \in \A^1 \times S_1 \ | \ z = z_j, \ t_j = 0 \}. \] Note that the $Z_{1j}$ are pairwise disjoint and let $Y_1$ be the blow-up of $\A^1 \times S_1$ along their union. The family $Y_1 \to S_1$ defines a smooth chart $S_1 \to \mathscr{M}_{\A^1}$ around $C_1$, since $H^0T_{C_1}(\mathscr{M}_{\A^1})$ is spanned by the $n$ deformations which move the nodes in $X$, along with the $n$ deformations which remove the components of $C_1 \setminus C_0$.

Since $Z_{1j}$ has a Zariski neighborhood isomorphic to a linear embedding $\A^{r-2} \cap V \to V$ for an open set $V \subset \A^r$, the normal bundle $N(Z_{1j}/\A^1 \times S_1)$ is trivial. In particular we can choose a coordinate $\pi_{1j} : E_{1j} \to \mathbb{P}^1$ on the exceptional divisor $E_{1j} = \mathbb{P}(N(Z_{1j}/\A^1 \times S_1))$ with the property that $\pi_{1j}^{-1}(\infty)$ intersected with each fiber over $S_1$ is contained in the singular locus. Let $n_2$ be the number of connected components of $C_2 \setminus C_1$ and put $S_2 = S_1 \times \A^{n_2}_{\disj} \times \A^{n_2}$. Let \[ \sigma_2 : \{ 1,\cdots,n_2 \} \to \{ 1,\cdots,n \} \] be a function such that $\sigma_2^{-1}(i)$ is the number of components of $C_2 \setminus C_1$ whose closure intersects the $i^{\text{th}}$ component of $C_1 \setminus C_0$. For each $1 \leq j \leq n_2$ define \[ Z_{2j} := \{ (y,(z_i),(t_i)) \in Y_1 \times \A^{n_2}_{\disj} \times \A^{n_2} \ | \ y \in E_{1\sigma_2(j)}, \pi_{1\sigma_2(j)}(y) = z_j, t_j = 0 \}, \] and let $Y_2$ be the blow-up of $Y_1 \times \A^{n_2}_{\disj} \times \A^{n_2}$ along the union of the $Z_{2j}$. Then $Y_2 \to \A^1 \times S_2$ defines a smooth chart $S_2 \to \mathscr{M}_{\A^1}$ around $C_2$.

Continue in this way to obtain a smooth chart $S_m \to \mathscr{M}_{\A^1}$ around $C$. At the inductive step, the center $Z_{ij}$ of each blow-up has a neighborhood of the form $\A^{r-2} \cap V \to V$ for an open $V \subset \A^r$, so that a suitable coordinate on the resulting exceptional divisor can be chosen. The preimage of $\partial \mathscr{M}_{\A^1}$ in $S_m$ equals the preimage of the union of the coordinate hyperplanes under the smooth projection $S_m \to \A^n$.

\end{proof}

The theorem now follows readily.

\begin{proof}[Proof of Theorem \ref{ulatwist}]

We apply Lemma \ref{keylem} with $\mathscr{Y} = \overline{\Bun}\!\,^{K,\lambda}_{B,\leq \mu}$ and $\mathscr{Z} = \Bun_G$. The boundary in $\overline{\Bun}\!\,^K_B$ is the preimage of the boundary in $\mathscr{M}_X$ under the canonical projection, so the boundary in the former is a normal crossings divisor by Propositions \ref{kontsm} and \ref{ncd}. Let $\mathscr{L}_{\overline{\Bun}\!\,^K_B}$ be the inverse image of $\mathscr{L}_{\overline{\Bun}_B}$. Since $\mathscr{L}_{\overline{\Bun}\!\,^K_B}$ trivializes canonically over $\Bun_B$, it has the form (\ref{twistmult}). Moreover, Proposition \ref{relsm} implies that the the strata in the boundary of $\overline{\Bun}\!\,^{K,\lambda}_{B,\leq \mu}$ determined by the normal crossings structure are smooth over $\Bun_G$, i.e. the hypotheses of Lemma \ref{keylem} are satisfied and hence $\IC^c_{\overline{\Bun}\!\,^{K,\lambda}_{B,\leq \mu}}$ is ULA over $\Bun_G$. Now apply the argument from the proof of Corollary \ref{ula} to deduce the same for $\IC^c_{\overline{\Bun}_{B,\leq \mu}^{\lambda}}$.

\end{proof}

\end{document}